\documentclass{amsart}

\newtheorem{theorem}{Theorem}[section]
\newtheorem{lemma}[theorem]{Lemma}
\newtheorem{proposition}[theorem]{Proposition}

\theoremstyle{definition}

\theoremstyle{remark}
\newtheorem{remark}[theorem]{Remark}

\numberwithin{equation}{section}

\usepackage{graphicx}

\usepackage[colorlinks=true, citecolor=blue, linkcolor=blue]{hyperref} 
\newcommand{\mku}{\mkern1mu}

\newcommand{\deq}{\mathrel{\mathop:}=}

\newcommand{\PP}{\mathbb{P}}
\newcommand{\EE}{\mathbb{E}}

\DeclareMathOperator{\ESF}{ESF}
\DeclareMathOperator{\Poi}{Poisson}
\DeclareMathOperator{\Gam}{Gamma}
\DeclareMathOperator{\PD}{PD}

\begin{document}

\title[On the Golomb--Dickman constant under Ewens sampling]{On the Golomb--Dickman constant \\ under Ewens sampling}

\author[J. R. G. Mendon\c{c}a]{Jos\'{e} Ricardo G. Mendon\c{c}a}
\address{Escola de Artes, Ci\^{e}ncias e Humanidades, Universidade de S\~{a}o Paulo, SP, Brazil.}
\email{jricardo@usp.br}
\thanks{JRGM was partially supported by research grant AP.R 2020/04475-7 from FAPESP, Brazil.}

\author[L. J. Negret]{Luis Jehiel Negret}
\address{Instituto de Matemática e Estat\'{\i}stica, Universidade de S\~{a}o Paulo, SP, Brazil}
\email{ljnegrett@ime.usp.br}
\thanks{LJN was partially supported by a postgraduate fellowship from CNPq, Brazil.}

\subjclass[2020]{Primary 60G55}

\keywords{Random permutations; probabilistic combinatorics; cycle structure; Ewens sampling formula; Poisson process; Kingman's construction}

\date{April 22, 2026}

\begin{abstract}
We define a generalized Golomb--Dickman constant $\lambda_{\theta}$ as the limiting expected proportion of the longest cycle in random permutations under the Ewens measure with parameter $\theta > 0$. Exploiting the independence properties of Kingman's Poisson process construction of the Poisson--Dirichlet distribution, we obtain an explicit integral representation for $\lambda_{\theta}$ in terms of the exponential integral. The dependence of $\lambda_{\theta}$ on $\theta$ reflects the transition between regimes dominated by long cycles (small $\theta$) and those with many small cycles (large $\theta$). We also derive the asymptotic behavior of $\lambda_{\theta}$ for small and large $\theta$ and illustrate our results with numerical computations, Monte Carlo simulations of the Hoppe urn, and an application.
\end{abstract}

\maketitle


\section{Introduction}
\label{sec:intro}

The cycle structure of random permutations is a classical, almost foundational topic in probabilistic combinatorics~\cite{Goncharov1944,Touchard1939}. For permutations of $[\mku{n}\mku] \deq \{1,\dots,n\} \subseteq \mathbb{N}$ drawn uniformly at random, a classical result of Shepp and Lloyd~\cite{SheppLLoyd1966} shows that the normalized length $L_{n}/n$ of the longest cycle converges to a non-degenerate random variable with expectation
\begin{equation}
\lim_{n \to \infty} \mathbb{E}\biggl[\frac{L_{n}}{n}\biggr] = \lambda,
\end{equation}
where $\lambda \approx 0.624330$ is the Golomb--Dickman constant.
The distribution of $L_n/n$ converges to a non-degenerate limit closely
related to the Dickman function~\cite{ABT2003}.

A natural extension of the uniform measure is the Ewens distribution with parameter $\theta>0$, which assigns probability proportional to $\theta^{C(\sigma)}$ to a permutation $\sigma$, where $C(\sigma)$ denotes its number of cycles~\cite{Crane2016,Ewens1972,Tavare2021}. This model arises in population genetics as the sampling distribution of allele frequencies under neutral evolution. Under the Ewens distribution, the numbers of cycles of fixed length converge asymptotically to independent Poisson random variables, and the normalized ordered cycle lengths converge to a Poisson--Dirichlet distribution $\PD(\theta)$ with parameter $\theta$~\cite{ABT2003,Arratia1992,Pitman2006}. The Poisson approximation for cycle counts, according to which the numbers of cycles of length $j$ behave approximately as independent $\Poi(\theta/j)$ variables, greatly simplifies the analysis of cycle statistics. Similar constructions appear in the study of the prime factorization of large integers, revealing a form of universality in the decomposition of large objects into smaller components~\cite{Ford2022,TerryTao2013}.

While the global structure of cycles is well established, explicit characterizations of extremal statistics remain scarce. The general Poisson--Dirichlet framework developed in~\cite{ABT2003,Arratia1992,Pitman2006} expresses everything in terms of the Dickman function (which lacks a closed-form solution in terms of elementary functions), the somewhat unyielding GEM distribution, or the Poisson process itself, besides all the combinatorics, and extracting computable integrals for cycle statistics can be challenging. Our goal is to evaluate the asymptotic expectation of the longest cycle under Ewens sampling using a more direct, probabilistic approach. We introduce the constant 
\begin{equation}
\label{eq:limlambda}
\lambda_{\theta} \deq \lim_{n \to \infty} \frac{\mathbb{E}_{n}^{\theta}[L_{n}]}{n},
\end{equation}
which extends the classical Golomb--Dickman constant, and derive an explicit integral representation for $\lambda_{\theta}$ in terms of the exponential integral. Such an integral representation for the limiting expected value \eqref{eq:limlambda} was previously derived in the continuous Poisson--Dirichlet setting by Holst~\cite{Holst2001} (see Remark~\ref{rmk:holst}). We recover this exact formula using Kingman's Poisson process construction of $\PD(\theta)$ and the distribution of its largest atom. This discrete-to-continuous derivation allows us to analyze $\lambda_{\theta}$, including its asymptotic behavior and combinatorial interpretations.

The paper is organized as follows. After briefly recalling the Poisson representation of Ewens permutations, we derive the limiting distribution of the largest atom of the Poisson process in the associated weighted model. We then establish an integral representation for $\lambda_{\theta}$ by exploiting the independence properties of Kingman's Poisson process construction of the Poisson--Dirichlet distribution $\PD(\theta)$, and extract the asymptotic behavior of $\lambda_{\theta}$ for small and large $\theta$. We conclude with numerical computations, Monte Carlo simulations of the Hoppe urn, and a combinatorial application to the spaghetti hoops problem.


\section{The Poisson representation of Ewens permutations}
\label{sec:poisson}

We briefly recall the Poisson representation of Ewens permutations~\cite{ABT2003,Arratia1992,Pitman2006}. Let $S_{n}$ denote the symmetric group on $[\mku{n}\mku]$, and for $\sigma \in S_{n}$, let $C_{j}(\sigma)$ be its number of cycles of length $j$ and $C(\sigma) = \sum_{j \ge 1} C_{j}(\sigma)$ its total number of cycles (note that $C_{j} = 0$ for $j>n$). The Ewens measure with parameter $\theta > 0$ over $S_{n}$ is given by
\begin{equation}
\mathbb{P}_{n}^{\theta}(\sigma) = 
\frac{\theta^{C(\sigma)}}{\theta(\theta+1)\cdots(\theta+n-1)},
\end{equation}
and the ensuing distribution, known as the Ewens sampling formula of parameter $\theta$, is denoted by $\ESF(\theta)$. The cycle counts in this model admit a Poisson approximation. For each fixed $j$, the distribution of $C_{j}(\sigma)$ is well approximated by
\begin{equation}
C_{j}(\sigma) \approx \Poi(\theta/j),
\end{equation}
and the variables $(C_{j})_{j \ge 1}$ are asymptotically independent in total variation when restricted to cycle lengths $j \le b(n)$ with $b(n) = o(n)$ or, equivalently, are restricted to finitely many cycle lengths~\cite{Arratia1992,Pitman2006}. A convenient way to formalize this approximation is through the construction of a suitable Poisson process.

Let $(N(t))_{t \ge 0}$ be a Poisson process of rate $1$, and let $(I_j)_{j \ge 1}$ be disjoint intervals in $[0,\infty)$ with $|I_j| = \theta/j$, $j \ge 1$. Define
\begin{equation}
\mu_{j} \deq N(I_{j}), \quad j \ge 1.
\end{equation}
The variables $(\mu_{j})_{j \ge 1}$ are independent and satisfy
\begin{equation}
\mu_{j} \sim \Poi(\theta/j).
\end{equation}
In terms of these variables, the Ewens distribution can be related to the law of $(\mu_j)_{j \ge 1}$ conditioned on the total size constraint
\begin{equation}
\label{eq:jmu}
\sum_{j \ge 1} j\mu_{j} = n.
\end{equation}
This representation translates questions about cycle structure to properties of independent Poisson variables. Variants of this construction arise in more general models of random permutations with nonuniform cycle weights as well~\cite{Ercolani2014}.

To approximate the constrained model in which \eqref{eq:jmu} holds, we replace the Poisson means $\theta/j$ by $\theta e^{-sj}/j$, where $s > 0$ is a conjugate parameter controlling the expected total size. This parameter shifts the unconstrained expectation $\EE[\mku\sum_{j \ge 1} j\mu_j\mku]$ to a value of order $1/s$, so that choosing $s \sim 1/n$ tunes the model to the correct scale. For $s>0$, we shall thus consider independent random variables $(\mu_{j})_{j \ge 1}$ with
\begin{equation}
\label{eq:mujesj}
\mu_{j}^{(s)} \sim \Poi\Bigl(\frac{\theta e^{-sj}}{j}\Bigr), \quad j \ge 1,
\end{equation}
This exponential tilting isolates the contribution of large cycles and facilitates the analysis of the longest cycle. Relaxing the size constraint to obtain independent cycle counts is analogous to the passage from the canonical to the grand-canonical ensemble in statistical mechanics~\cite{Campisi2024}.

\begin{remark}
\label{rmk:poissonmodel}
One must distinguish the actual cycle counts $C_{j}$, which are strictly dependent due to the partition constraint $\sum_{j \ge 1} jC_{j} = n$, from the unconstrained auxiliary variables $\mu_{j}$.  The independent tilted Poisson model $(\mu_{j}^{(s)})_{j \ge 1}$ circumvents the constraint, allowing the longest cycle under the Ewens measure to be analyzed directly via the largest occupied index of the unconstrained sequence and the independence properties of Kingman's Poisson process construction (see Section~\ref{sec:lambda}).
\end{remark}


\section{The scaling limit of the largest Poisson atom}
\label{sec:scaling}

We now determine the distribution of the largest occupied index in the tilted Poisson model. As shown in Section~\ref{sec:lambda}, this coincides with the distribution of the largest atom in a Kingman Poisson process, yielding $\lambda_{\theta}$.

Let $(\mu_{j}^{(s)})_{j \ge 1}$ be independent random variables distributed as in \eqref{eq:mujesj}, and define the largest occupied index in the tilted model as
\begin{equation}
L^{(s)} \deq \max\{j \ge 1 \colon \mu_{j}^{(s)} > 0\}.
\end{equation}
The distribution of $L^{(s)}$ can be computed explicitly and admits a scaling limit.

\begin{lemma}
\label{lemma:largest}
For each $k \ge 1$, we have
\begin{equation}
\mathbb{P}(L^{(s)} < k)
= \exp\biggl(-\sum_{j \ge k} \frac{\theta e^{-sj}}{j}\biggr).
\end{equation}
Moreover, under the scaling $x = sk$, we have
\begin{equation}
\mathbb{P}(s L^{(s)} \le x)
\;\longrightarrow\;
\exp\bigl(-\theta E_{1}(x)\bigr) \text{ as } s \to 0,
\end{equation}
where
\begin{equation}
 E_{1}(x) = \int_{x}^{\infty} \frac{e^{-t}}{t}\,dt
\end{equation}
is the exponential integral function of a nonnegative real argument \cite[{\S}15.09]{Jeffreys1972}.
\end{lemma}
\begin{proof}
By independence,
\begin{equation}
\mathbb{P}(L^{(s)} < k)
= \prod_{j \ge k} \mathbb{P}(\mu_{j}^{(s)} = 0)
= \exp\biggl(-\sum_{j\ge k} \frac{\theta e^{-s j}}{j}\biggr).
\end{equation}
A standard sum--integral comparison for the monotone function $x \mapsto e^{-s x}/x$ yields
\begin{equation}
\sum_{j\ge k} \frac{e^{-s j}}{j}
= \int_{s k}^\infty \frac{e^{-u}}{u}\,du + o(1) \text{ as } s \to 0,
\end{equation}
which gives the stated limit.
\end{proof}

The function $\exp[-\theta E_{1}(x)]$ is the probability that a Poisson process on $(0,\infty)$ with intensity $\theta e^{-t}/t$ places no atoms above~$x$; it is, therefore, the distribution function of the largest atom of this process. For general $\theta > 0$, it approximates but does not quite equal the distribution of the largest $\PD(\theta)$ component; the discrepancy comes from the normalization by the total mass. Despite this somewhat subtle difference, the distribution of the largest atom suffices to compute the expectation of the largest normalized cycle via Kingman's construction in the next section.


\section{Integral representation of $\lambda_{\theta}$ and its asymptotics}
\label{sec:lambda}

\subsection{The main theorem}
\label{subsec:main}

Our main result, an integral formula for $\lambda_{\theta}$, is given by the following theorem.

\begin{theorem} 
Let $L_{n}$ denote the length of the longest cycle of a random permutation distributed according to the Ewens measure with parameter $\theta>0$. Then the limit
\begin{equation}
\lambda_{\theta} 
\deq \lim_{n \to \infty} \frac{\mathbb{E}_{n}^{\theta}[L_{n}]}{n}
\end{equation}
exists and is given by
\begin{equation}
\lambda_{\theta}
= \int_{0}^{\infty} \exp\bigl[-t-\theta E_{1}(t)\bigr]\,dt.
\end{equation}
\end{theorem}
\begin{proof}
We compute the expected value of the longest cycle of a random permutation under the Ewens measure by exploiting the independence properties of Kingman's Poisson process construction of the Poisson--Dirichlet process $\PD(\theta)$~\cite{Kingman1975,Kingman1993}. The argument relies on the independence of the normalized partition from the total mass, together with the distribution of the largest atom from Lemma~\ref{lemma:largest}.

Let $L^{(s)}$ denote the largest occupied index in the tilted Poisson model. Lemma~\ref{lemma:largest} shows that, under the scaling $x = s k$,
\begin{equation}
\mathbb{P}(sL^{(s)} \le x) \longrightarrow \exp\bigl(-\theta E_{1}(x)\bigr),
\end{equation}
so that $sL^{(s)}$ converges in distribution to a random variable $X$
with distribution function
\begin{equation}
\mathbb{P}(X \le x) = \exp\bigl(-\theta E_{1}(x)\bigr).
\end{equation}
Independently, if $\Pi$ is a Poisson process on $(0,\infty)$ with
intensity $\theta e^{-t}/t$ and $X_{1} > X_{2} > \cdots$ denote its
atoms in decreasing order, then
\begin{equation}
\PP(X_{1} \le x) = \PP(\Pi((x,\infty))=0) = \exp(-\theta E_{1}(x)),
\end{equation}
so the limiting distribution of $sL^{(s)}$ coincides with the
distribution of the largest atom $X_{1}$ of $\Pi$. No process-level
convergence is required; the identification is purely at the level of
distribution functions.

This distribution, however, does not coincide pointwise with that of the largest component of the Poisson--Dirichlet partition due to the normalization by the total mass. Denote by $Y_{\theta}$ the largest component of a $\PD(\theta)$ random partition. The convergence of the normalized cycle lengths to $\PD(\theta)$ implies $L_{n}/n \to Y_{\theta}$ in distribution, and since $L_{n}/n \le 1$, dominated convergence gives $\lambda_{\theta} = \EE[Y_{\theta}]$.

To compute $\EE[Y_{\theta}]$, let $\Sigma = \sum_{i}X_{i}$ denote the total mass of $\Pi$. The normalized sequence $(X_{1}/\Sigma, X_{2}/\Sigma, \dots)$ has the $\PD(\theta)$ distribution, and $\Sigma$ is independent of the normalized partition, with $\Sigma \sim \Gam(\theta,1)$~\cite{Kingman1975,Kingman1993} (see also \cite[\S3.2]{Pitman2006} and~\cite{TerryTao2013}). Since $Y_{\theta} = X_{1}/\Sigma$ and $\Sigma$ is independent of $Y_{\theta}$, we have $\EE[X_{1}] = \EE[\Sigma]\mku\EE[Y_{\theta}] = \theta\lambda_{\theta}$. The tail expectation formula for non-negative random variables then gives
\begin{equation}
\EE[X_{1}] = \int_{0}^{\infty} \bigl[1-\exp[-\theta E_{1}(x)]\bigr]\,dx.
\end{equation}
Now, since $E_{1}'(x) = -e^{-x}/x$ we have
\begin{equation}
1-\exp[-\theta E_{1}(x)] = 
\theta \int_{x}^{\infty} \frac{e^{-t}}{t}\exp[-\theta E_{1}(t)]\,dt.
\end{equation}
Substituting this into the expression for $\EE[X_{1}]$, integrating over $x$, and swapping the integrations in $dx$ and $dt$ (the integrand is non-negative) furnishes
\begin{equation}
\EE[X_{1}] = \theta \int_{0}^{\infty} e^{-t} \exp[-\theta E_{1}(t)]\,dt,
\end{equation}
which divided by $\theta$ finally yields
\begin{equation}
\label{eq:ggd}
\lambda_{\theta} = \EE[Y_{\theta}] = \int_{0}^{\infty} \exp[-t-\theta E_{1}(t)]\,dt.
\end{equation}
\end{proof}

Since $E_{1}(t)$ is positive and decreasing, the integrand $\exp\bigl[-t-\theta E_{1}(t)\bigr]$ is decreasing in $\theta$ for each fixed $t>0$, and therefore $\lambda_{\theta}$ is decreasing in $\theta$. For $\theta<1$, the contribution of large values of $t$ is less suppressed, leading to larger values of $\lambda_{\theta}$ and a regime dominated by long cycles. When $\theta=0$, $\lambda_{0}=1$, meaning that the longest cycle is also the only cycle in the permutation. For $\theta > 1$, the factor $\exp[-\theta E_{1}(t)]$ decays more rapidly, resulting in smaller values of $\lambda_{\theta}$ and a regime where the mass is distributed among many smaller cycles. As $\theta \to \infty$, the Ewens measure increasingly favors permutations with many short cycles, and the longest cycle becomes negligible relative to $n$. This agrees with known behavior under the Ewens distribution and the Poisson–Dirichlet framework \cite{ABT2003,Arratia1992,Ercolani2014,Ford2022,Holst2001,Pitman2006}.

\begin{remark}
\label{rmk:didnot}
We do not address the rate of convergence of $\EE_{n}^{\theta}[L_{n}]/n$ to its limit or the full distribution of $L_{n}/n$. Quantitative bounds on the total variation distance between the cycle counts and their Poisson approximation are available \cite{Arratia1992,Ford2022} and could be used to estimate the convergence rate, but we do not pursue this here.
\end{remark}

\begin{remark}
\label{rmk:holst}
Lars Holst obtained, among other results, an integral expression for the $k$-th moment of the largest $\PD(\theta)$ component that reads \cite[Proposition~2.2]{Holst2001}
\begin{equation}
\label{eq:holst}
\EE[Y_{\theta}^{k}] = \int_{0}^{\infty} \dfrac{y^{k-1} \exp[-y-\theta E_{1}(y)]}{(\theta+1)\cdots(\theta+k-1)}\,dy.
\end{equation}
Although equation \eqref{eq:ggd} is the special case $k=1$ of Holst's equation, our derivation and subsequent analysis differ from Holst's work. Methodologically, our derivation proceeds via the scaling limit of the largest occupied index in the tilted Poisson model (Lemma~\ref{lemma:largest}) followed by an exchange of integrations that yields the closed-form integrand $\exp\bigl[-t-\theta E_{1}(t)\bigr]$ directly, whereas Holst integrates the density of the largest atom against powers of its argument within a finite-dimensional Dirichlet framework. In terms of scope, we emphasize the systematic study of $\lambda_{\theta}$ as a one-parameter extension of the classical Golomb--Dickman constant and analyze its monotonicity in~$\theta$, asymptotic behavior in the regimes $\theta \to 0$ and $\theta \to \infty$ (Section~\ref{sec:asymptotics}), numerical computation and simulation (Sections~\ref{subsec:numbers} and \ref{subsec:hoppe}), and possible combinatorial interpretations. None of these aspects appear in Holst's work.
\end{remark}


\subsection{Asymptotic behavior of $\lambda_{\theta}$}
\label{sec:asymptotics}

In this section we extract the leading asymptotic behavior of $\lambda_{\theta}$ in the regimes of $\theta \to 0$ and $\theta \to \infty$ directly from the integral representation \eqref{eq:ggd}.

\begin{proposition}
\label{prop:smalltheta}
As $\theta \to 0$,
\begin{equation}
\label{eq:smalltheta}
\lambda_{\theta} = 1 - (\ln{2})\theta + O(\theta^{2}).
\end{equation}
\end{proposition}
\begin{proof}
Expanding $\exp[-\theta E_{1}(t)] = 1-\theta E_{1}(t) + O(\theta^{2})$ in~\eqref{eq:ggd} and integrating term by term gives
\begin{equation}
\begin{split}
\lambda_{\theta} &= \int_{0}^{\infty} e^{-t}\,dt 
-\theta\int_{0}^{\infty} e^{-t} E_{1}(t)\,dt + O(\theta^{2}) \\
& = 1 - \theta \int_{0}^{\infty} e^{-t} E_{1}(t)\,dt + O(\theta^{2}).
\end{split}
\end{equation}
Substituting the definition of $E_{1}(t)$ and swapping the order of integration yields for the remaining integral above
\begin{equation}
\label{eq:firstorder}
\begin{split}
\int_{0}^{\infty} e^{-t} E_{1}(t)\,dt 
&= \int_{0}^{\infty} e^{-t} \int_{t}^{\infty} \frac{e^{-u}}{u}\,du\,dt \\
&= \int_{0}^{\infty} \frac{e^{-u}}{u} \int_{0}^{u} e^{-t}\,dt\,du 
= \int_{0}^{\infty} \frac{1 - e^{-u}}{u}\,e^{-u}\,du.
\end{split}
\end{equation}
If we write $(1-e^{-u})/u = \int_{0}^{1} e^{-uv}\,dv$ and integrate \eqref{eq:firstorder} in~$du$ first we get $\ln{2}$, and we are done.
\end{proof}

\begin{proposition}
\label{prop:largetheta}
As $\theta \to \infty$,
\begin{equation}
\label{eq:largetheta}
\lambda_{\theta} = \frac{\ln{\theta}}{\theta} - \frac{\ln{\ln{\theta}} - \gamma}{\theta} + o\Bigl(\frac{1}{\theta}\Bigr),
\end{equation}
where $\gamma \approx 0.577215$ is Euler's constant.
\end{proposition}
\begin{proof}
As $\theta \to \infty$, the integral \eqref{eq:ggd} is dominated by large values of $t$, where the exponential integral behaves as $E_{1}(t) \sim e^{-t}/t$. Since $u = \theta E_{1}(t) \approx \theta e^{-t}/t$ for large $t$, the relation $\theta e^{-t}/t = u$ gives $du = -\theta e^{-t} (t+1)/t^{2}\,dt \approx -u\,dt$. Inverting the asymptotic relation between $u$ and $t$ gives $t = \ln{\theta} - \ln{\ln{\theta}} - \ln{u} + o(1)$, which substituted into the integral furnishes
\begin{equation}
    \lambda_\theta \approx \int_0^\infty \frac{t}{\theta} e^{-u} \, du = \frac{1}{\theta} \int_{0}^{\infty} (\ln{\theta} - \ln{\ln{\theta}} - \ln{u})\mku e^{-u}\,du.
\end{equation}
Evaluating the standard integral $\int_{0}^{\infty}e^{-u}\ln{u}\,du = -\gamma$ yields the stated result.
\end{proof}

Note that the decay of $\ln{\theta}/\theta$ is slower than $1/\theta$. Since $1/\theta$ would be the expected proportion of the longest cycle if all cycles had equal length, the logarithmic correction reflects the fact that random fluctuations persistently produce one cycle that captures a disproportionate share of the total mass, even when the Ewens measure strongly favors many small cycles. Note also that the appearance of Euler's constant in \eqref{eq:largetheta} reflects the logarithmic growth of the harmonic sums that govern the cycle count distribution.


\section{Numerics, simulations and applications}
\label{sec:numberics}

\subsection{Numerical values}
\label{subsec:numbers}

Figure~\ref{fig:ggd} displays the behavior of $\lambda_{\theta} \times \theta$ for $0 \le \theta \le 5$, while Table~\ref{tab:ggd} lists $\lambda_{\theta}$ for a few selected $\theta$. These values were calculated in Python to a precision of $\approx 10^{-12}$ from the routines \texttt{scipy.integrate.quad} and \texttt{scipy.special.exp1}. The particular value ${\lambda_{\theta}=0.5}$ appears at ${\theta \simeq 1.784910}$.

\begin{figure}[ht]
\centering
\includegraphics[scale=0.55, clip]{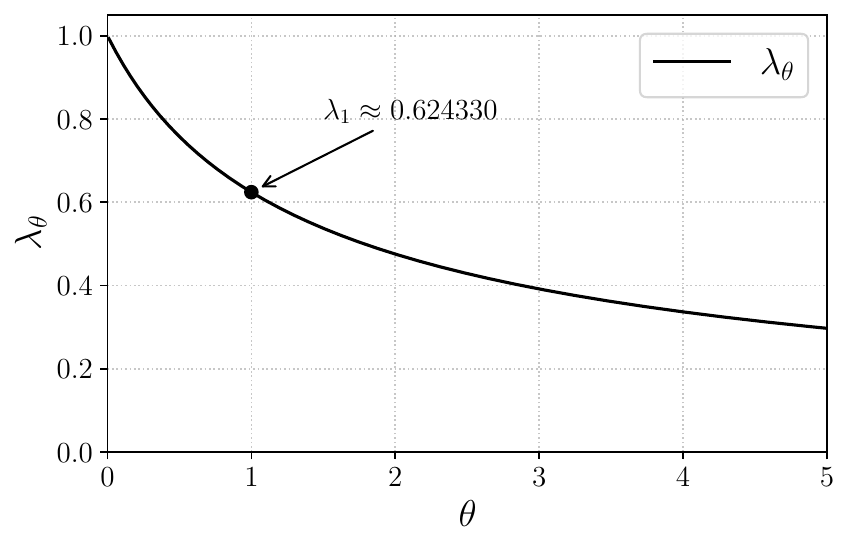}
\caption{Generalized Golomb--Dickman constant $\lambda_{\theta}$, equation \eqref{eq:ggd}.}
\label{fig:ggd}
\end{figure}

\begin{table}[ht]
\caption{Numerical value of $\lambda_{\theta}$ (rounded to $6$ decimal places) for a few selected $\theta$.}
\label{tab:ggd}
\centering
\begin{minipage}{0.30\textwidth}
\centering
    \begin{tabular}{cc}
    \hline \hline
    $\theta$ & $\lambda_{\theta}$ \\ 
    \hline
    $1/10$   & $0.936295$ \\
    $1/8$    & $0.921937$ \\
    $1/6$    & $0.899210$ \\
    $1/5$    & $0.882027$ \\
    $1/4$    & $0.857758$ \\
    $1/3$    & $0.820854$ \\
    $1/2$    & $0.757823$ \\ 
    $2/3$    & $0.705779$ \\
    \hline \hline
    \end{tabular}
\end{minipage}
\begin{minipage}{0.30\textwidth}
\centering
    \begin{tabular}{cc}
    \hline \hline
    $\theta$ & $\lambda_{\theta}$ \\ 
    \hline
    $3/4$    & $0.682960$ \\
    $1$      & $0.624330$ \\
    $3/2$    & $0.537540$ \\
    $2$      & $0.475639$ \\
    $3$      & $0.391838$ \\
    $4$      & $0.336771$ \\
    $5$      & $0.297288$ \\ 
    $10$     & $0.194884$ \\ 
    \hline \hline
    \end{tabular}
\end{minipage}
\end{table}

Formula \eqref{eq:ggd} gives the expected proportion of the longest component in the ``spaghetti hoops problem''~\cite{Tavare2021}. In this problem, one starts with $n$ strands of spaghetti and repeatedly picks two free ends uniformly at random and ties them together until no free ends remain, producing a collection of closed loops. This process gives rise to the Ewens distribution with parameter $\theta = 1/2$, and the expected proportion of the total length contained in the longest loop is therefore given by $\lambda_{1/2}$ in the limit $n \to \infty$. From Table~\ref{tab:ggd}, we see that approximately $75.8\%$ of the spaghetti will end up tied together in one big loop. 

\subsection{Simulation of the Hoppe urn model}
\label{subsec:hoppe}

The Hoppe urn provides a stochastic mechanism to generate random partitions whose distribution coincides with the Ewens sampling formula ~\cite{Aldous1985,Hoppe1984}. The process starts with a distinguished, say, black ball of weight $\theta$. At each step, a ball is selected with probability proportional to its weight. If the black ball is chosen, a new color is introduced; otherwise, the selected color is reinforced. After $n$ steps, the resulting color classes define a random partition of $[\mku{n}\mku]$ according to the $\ESF(\theta)$ distribution. 

An algorithmic description of the Hoppe urn process goes as follows:
\begin{itemize}
\item[1.] Initialize one color class and set its size to zero.
\item[2.] For $k = 1, \dots, n$,
\begin{itemize}
\item[a.] With probability $\theta/(\theta + k - 1)$, create a new color class;
\item[b.] Otherwise, choose an existing class with probability proportional to its size and add one element to it.
\end{itemize}
\item[3.] Return the class sizes $n_{1}, \dots, n_{\ell}$, $1 \le \ell \le n$.
\end{itemize}
Figure~\ref{fig:singlerun} displays the time evolution of the color proportions in one simulation of $50$ draws in a Hoppe urn with $\theta=1$.

\begin{figure}[h]
\centering
\includegraphics[scale=0.55, clip]{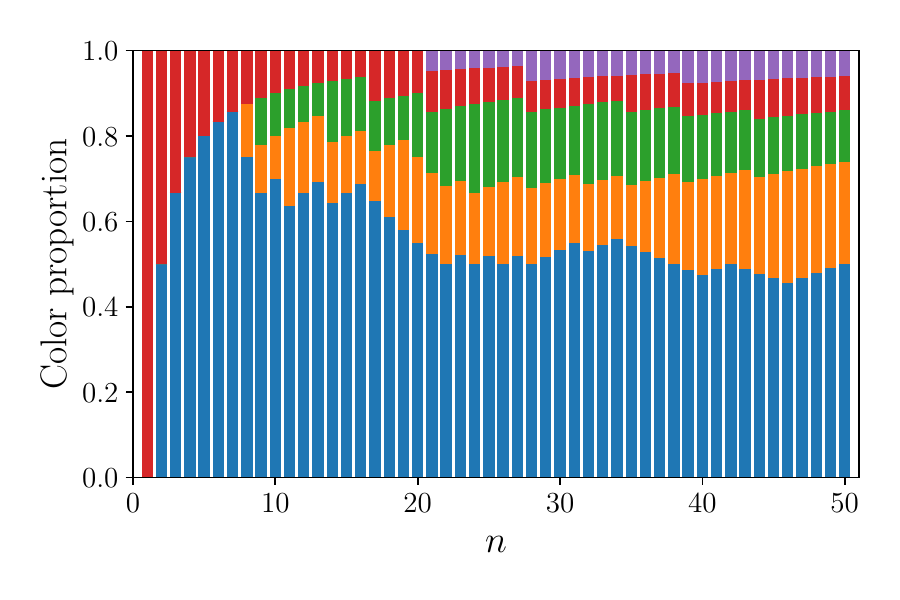}
\caption{Timeline of the color proportions in one simulation of $50$ draws in a Hoppe urn with $\theta=1$.}
\label{fig:singlerun}
\end{figure}

\begin{remark}
\label{rmk:chinese}
The Hoppe urn is closely related to the Chinese restaurant process (CRP), another sequential construction that generates permutations with the $\ESF(\theta)$ distribution~\cite{Aldous1985,Pitman2006}. In the CRP, customers labeled $1, 2, \ldots, n$ enter a restaurant one at a time; customer~$k$ either joins an occupied table with probability proportional to its occupancy, or opens a new table with probability $\theta/(\theta+k-1)$. The two constructions are equivalent: drawing the black ball in the Hoppe urn corresponds to opening a new table, and drawing a numbered ball to joining an existing one. We refer to \cite{Crane2016} for an overview of both the Hoppe urn and the CRP and their connections to the Ewens sampling formula.
\end{remark}

The simulation of the Hoppe urn model from the description given above is immediate, and the size of the largest class in the urn, normalized by $n$, provides a Monte Carlo estimator for $\lambda_{\theta}$.  Repeating the simulation many times and averaging the results yields empirical estimates of the asymptotic expected proportion of the largest cycle. Figure~\ref{fig:mcarlo} displays the Monte Carlo estimates of $\lambda_{\theta}$ for several values of $\theta$ obtained from averaging $10000$ runs of the Hoppe urn up to $n=1000$ draws, together with the exact value of $\lambda_{\theta}$ from \eqref{eq:ggd} and the asymptotic curves \eqref{eq:smalltheta} and \eqref{eq:largetheta}. The simulations illustrate the dependence of $\lambda_{\theta}$ on $\theta$ and agree with the theoretical predictions.

\begin{remark}
The empirical averages obtained from the simulations are subject to standard Monte Carlo sampling errors. For the sample sizes considered to produce Figure~\ref{fig:mcarlo}, these statistical fluctuations are negligible and therefore omitted, as the numerical data are included for illustration only.
\end{remark}

\begin{figure}[h]
\centering
\includegraphics[scale=0.55, clip]{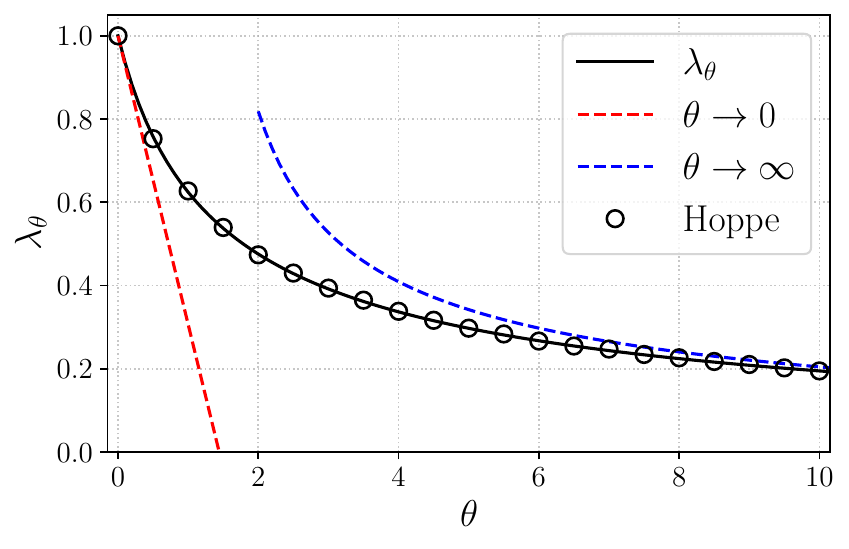}
\caption{Monte Carlo estimates of $\lambda_{\theta}$ obtained via the
Hoppe urn simulation compared with the theoretical curves.}
\label{fig:mcarlo}
\end{figure}



\begin{thebibliography}{00}

\bibitem{Aldous1985}
D. J. Aldous,
\textit{Exchangeability and related topics},
In: P. L. Hennequin (Ed.), \'{E}cole d'Et\'{e} de Probabilit\'{e}s de Saint-Flour XIII, 1983, Lecture Notes in Mathematics \textbf{1117},
Springer, Berlin, 1985, pp.~1--198.
\MR{0883646}

\bibitem{ABT2003}
R. Arratia, A. D. Barbour, and S. Tavar\'{e},
\textit{Logarithmic Combinatorial Structures: A Probabilistic Approach},
European Mathematical Society, Z\"{u}rich, 2003. 
\MR{2032426}

\bibitem{Arratia1992}
R. Arratia and S. Tavar{\'e},
\textit{The cycle structure of random permutations},
{The Annals of Probability} \textbf{20} (1992), no.~3, 1567--1591.
\MR{1175278}

\bibitem{Campisi2024}
M. Campisi,
\textit{Lectures on the Mechanical Foundations of Thermodynamics}, 2nd ed.
Springer, Cham, 2024.

\bibitem{Crane2016}
H. Crane, 
\textit{The ubiquitous Ewens sampling formula},
{Statistical Science} \textbf{31} (2016), no.~1, 1--19.
\MR{3458585}

\bibitem{Ercolani2014}
N. M. Ercolani and D. Ueltschi,
\textit{Cycle structure of random permutations with cycle weights},
{Random Structures \&\ Algorithms} \textbf{44} (2014), no.~1, 109--133.
\MR{3143592}

\bibitem{Ewens1972}
W. J. Ewens,
\textit{The sampling theory of selectively neutral alleles},
{Theoretical Population Biology} \textbf{3} (1972), 87--112.
\MR{0325177}

\bibitem{Ford2022}
K. Ford,
\textit{Cycle type of random permutations: A toolkit},
{Discrete Analysis} (2022), art.~9 (36pp). 
\MR{4481406}

\bibitem{Goncharov1944}
V. Gon\v{c}arov, 
\textit{On the field of combinatory analysis},
In: Twelve Papers on Number Theory and Function Theory,
{American Mathematical Society Translations (2)} \textbf{19} (1962), 1--46. \MR{0131369}
[Originally published in {Izvestiya Akademii Nauk SSSR, Seriya Matematicheskaya} \textbf{8} (1944), no.~1, 3--48, in Russian. \MR{0010922}]
    
\bibitem{Holst2001}
L. Holst, 
\textit{The Poisson-Dirichlet distribution and its relatives revisited},
preprint (2001), Royal Institute of Technology, Stockholm, Sweden, URL: \url{https://www.math.kth.se/matstat/fofu/reports/PoiDir.pdf}.

\bibitem{Hoppe1984}
F. M. Hoppe,
\textit{P{\'o}lya-like urns and the {E}wens' sampling formula},
{Journal of Mathematical Biology} \textbf{20} (1984), no.~1, 91--94.

\bibitem{Jeffreys1972}
H. Jeffreys and B. S. Jeffreys, 
\textit{Methods of Mathematical Physics}, 3rd~ed.,
Cambridge University Press, Cambridge, 1972.
\MR{1744997}

\bibitem{Kingman1975}
J. F. C. Kingman, 
\textit{Random discrete distributions},
{Journal of the Royal Statistical Society. Series B (Methodological)} \textbf{37} (1975), no.~1, 1--22.
\MR{0368264}

\bibitem{Kingman1993}
J. F. C. Kingman, 
\textit{Poisson Processes},
Oxford Studies in Probability \textbf{3}, 
Oxford University Press, Oxford, 1993.
\MR{1207584}

\bibitem{Pitman2006}
J. Pitman, 
\textit{Combinatorial Stochastic Processes}, In: J. Picard (Ed.), \'{E}cole d'Et\'{e} de Probabilit\'{e}s de Saint-Flour XXXII, 2002, Lecture Notes in Mathematics \textbf{1875},
Springer, Heidelberg, 2006.
\MR{2245368}

\bibitem{SheppLLoyd1966}
L. A. Shepp and S. Lloyd,
\textit{Ordered cycle lengths in a random permutation},
{Transactions of the American Mathematical Society} \textbf{121} (1966), 340--357.
\MR{195117}

\bibitem{TerryTao2013}
T. Tao,
\textit{The Poisson--Dirichlet process, and large prime factors of a random number},
{What's New Daily Archive}, URL: \url{https://terrytao.wordpress.com/2013/09/21/}. 

\bibitem{Tavare2021}
S. Tavar{\'e}, 
\textit{The magical {E}wens sampling formula},
{Bulletin of the London Mathematical Society} \textbf{53} (2021), no.~6, 1563--1582.
\MR{4368686}

\bibitem{Touchard1939}
J. Touchard, 
\textit{Sur les cycles des substitutions},
{Acta Mathematica} \textbf{70} (1939), no.~1, 243--297.
\MR{1555449}

\end{thebibliography}
\end{document}